\documentclass{amsart}
\usepackage{amsmath,amsfonts,amsthm,amssymb,indentfirst,epic,xurl,graphics,needspace}

\newtheorem{theorem}{Theorem}
\newtheorem{lemma}[theorem]{Lemma}
\newtheorem{corollary}[theorem]{Corollary}
\newtheorem{proposition}[theorem]{Proposition}
\newtheorem{definition}[theorem]{Definition}

\numberwithin{equation}{section}
\numberwithin{theorem}{section}

\renewcommand{\r}{\mathrm}

\begin{document}

\begin{center}
\texttt{Comments, corrections,
and related references welcomed, as always!}\\[.5em]
{\TeX}ed \today
\\[.5em]
\vspace{2em}
\end{center}

\title%
{Further thoughts on dimensions of posets}
\thanks{%
Readable at \url{http://math.berkeley.edu/~gbergman/papers/unpublished},
or, if I decide to submit it for
publication, at \url{http://math.berkeley.edu/~gbergman/papers/}.
}

\subjclass[2020]{Primary: 06A07.
%                         combinatorics_of_posets
% Secondary:
}

\author{George M.\ Bergman}
\address{Department of Mathematics\\
University of California\\
Berkeley, CA 94720-3840, USA}
\email{gbergman@math.berkeley.edu}

\begin{abstract}
We recall the concept of the dimension of a finite poset $P,$
and the longstanding conjecture that for all finite
nonempty posets $P$ and $Q,$
$\dim(P\times Q)\geq\dim(P)+\dim(Q)-2.$
We then note two other plausible
inequalities, either of which would imply that one.

In the final section, writing $P\preccurlyeq P'$
if, for all $Q,$ $\dim(P\times Q)\leq\dim(P'\times Q),$
and writing $P\approx P'$ if $P\preccurlyeq P'$ and $P'\preccurlyeq P,$
we note some results and questions concerning these relations.
\end{abstract}
\maketitle

\section{Background}\label{S.bkgrd}

This note is an addendum to~\cite{poset_dim}.
Let us recall the basic concepts considered there, and set some
conventions we will follow.

\begin{definition}\label{D.main}
In this note, a {\em poset} will mean a {\em finite} {\em nonempty}
partially ordered set.
% {\rm(}These assumptions will occasionally be mentioned
% for emphasis{\rm.)}

The {\em dimension} $\dim(P)$ of a poset $P$ means the least
nonnegative integer $d$ such that $P$ can be embedded in
a direct product of $d$ chains.

A poset $P$ will be called {\em connected} if $P$ cannot
be written as the union $P_0\cup P_1$ of two nonempty subsets
such that no element of $P_0$ is comparable with any
member of $P_1$ under the partial ordering of~$P.$

\end{definition}

In~\cite{poset_dim}, posets were allowed to be infinite, though
the results there concerned posets of finite {\em dimension.}
As noted in~\cite[\S8]{poset_dim}, the dimension of a finite-dimensional
infinite poset is the supremum of the dimensions of
its finite subposets.
Given that fact, it will suffice for the purposes of this note to
restrict attention to finite posets, as specified
in the above Definition.

\section{Products of posets with added tops and bottoms}\label{S.P+}

A longstanding question in the theory of poset dimensions is
\begin{equation}\begin{minipage}[c]{35pc}\label{d.leq2?}
\cite[Question 3.2(i)]{poset_dim}\ \ %
If $P$ and $Q$ are posets, must
$\dim(P\times Q)\,\geq\,\dim(P)+\dim(Q)-2$?
\end{minipage}\end{equation}

Two related questions are
\begin{equation}\begin{minipage}[c]{35pc}\label{d.^+^+?}
\cite[Question 3.3(ii)]{poset_dim}\ \ %
If $P$ and $Q$ are posets each of which has a least element,
or each of which has a greatest element,
must $\dim(P\times Q)\,\geq\,\dim(P)+\dim(Q)-1$?
\end{minipage}\end{equation}
\begin{equation}\begin{minipage}[c]{35pc}\label{d._+^+?}
\cite[Question 3.3(i)]{poset_dim}\ \ %
If $P$ and $Q$ are posets
one of which has a least element and the other a greatest element,
must $\dim(P\times Q)\,\geq\,\dim(P)+\dim(Q)-1$?
\end{minipage}\end{equation}

We will note below a possible approach to these three questions.

First, some notation:

\begin{definition}\label{D.P+}
If $P$ is a poset, $P^+$ will denote the poset obtained
by adjoining to $P$ a single element greater than
all elements of $P,$ and $P_+$ the poset obtained
by adjoining to $P$ a single element less than all elements of $P.$
These added elements will be denoted $1$ and $0$ respectively,
unless existing elements of $P$ have already been so named.

The posets $(P^+)_+$ and $(P_+)^+$ will be identified with
each other, and the resulting poset denoted $P{_+^+}.$

Since the construction taking a poset $P$ to its opposite
{\rm(}the poset with the same underlying set but the reverse
ordering{\rm)} respects dimension and products, but interchanges
the operations  $(\ )^+$ and  $(\ )_+,$ we shall take for granted
that general results, questions, etc.\ stated for one of the latter two
operations are equivalent to the corresponding statements for the other.
\end{definition}

We note that
\begin{equation}\begin{minipage}[c]{35pc}\label{d.dim=}
If $P$ is a poset of positive dimension (equivalently, a poset
with more than one element), then
$\dim(P)=\dim(P^+)=\dim(P_+)=\dim(P{_+^+}).$
\end{minipage}\end{equation}
Clearly, it suffices to verify the first equality (cf.~last
paragraph of the above Definition).
This is the case of~\cite[(1.11)]{poset_dim} with the $\!2\!$-element
chain ${\bf 2}=\{0,1\}$ in the role of the indexing poset $P$
of that result, the poset here called $P$ in the role of $Q_0$
of that result, and the singleton poset $\{1\}$ in the role of $Q_1.$

(In fact, the brief proof of~\cite[(1.11)]{poset_dim}
comes down, in this case,
to taking a representation of the $P$ of~\eqref{d.dim=}
as a subposet of a product
of $\dim(P)$ chains, adjoining a new top element to
each of those chains, and representing $P^+$ as the subposet of the
product of these extended chains consisting of the
elements of $P,$ and the element having for all its coordinates
these new top elements.)

We recall
\begin{equation}\begin{minipage}[c]{35pc}\label{d.KB}
\cite[p.\,9, Property~3]{KB},
\cite[Theorem 3.1]{poset_dim},
\cite[p.\,179, last 11 lines]{K+T}.
If $P$ and $Q$ are posets each of which has {\em both} a least and a
greatest element, then $\dim(P\times Q)=\dim(P) + \dim(Q).$
\end{minipage}\end{equation}

Though as noted in~\eqref{d.dim=},
adding a top or bottom element to a poset of positive dimension
leaves that dimension unchanged, things can change when one
takes direct products of such enlarged posets.

By how much?
I know of no counterexamples to the bounds suggested in
the following questions.
\begin{equation}\begin{minipage}[c]{35pc}\label{d.P^+xQ^+?}
If $P$ and $Q$ are posets of cardinality~$>1,$
must $\dim(P^+\times Q^+)\,\leq\,\dim(P\times Q) + 1$?
\end{minipage}\end{equation}
\begin{equation}\begin{minipage}[c]{35pc}\label{d.P_+xQ^+?}
If $P$ and $Q$ are posets of cardinality~$>1,$
must $\dim(P_+\times Q^+)\,\leq\,\dim(P\times Q) + 1$?
\end{minipage}\end{equation}

These two questions are, in fact, related to those recalled
at the beginning of this section:

\begin{proposition}\label{P.=>}
{\rm(i)} A positive answer to
either~\eqref{d.P^+xQ^+?} or~\eqref{d.P_+xQ^+?}
would imply a positive answer to~\eqref{d.leq2?}.

{\rm(ii)} \,A positive answer to~\eqref{d.P^+xQ^+?}
would also imply a positive answer to~\eqref{d.^+^+?}.

{\rm(iii)} A positive answer to~\eqref{d.P_+xQ^+?}
would also imply a positive answer to~\eqref{d._+^+?}.
\end{proposition}

\begin{proof}
If $P$ and/or $Q$ is a singleton (equivalently, has dimension $0)$
then \eqref{d.leq2?}, \eqref{d.^+^+?}, and~\eqref{d._+^+?}
trivially have affirmative answers, so let us assume we are given
$P$ and $Q$ both of cardinality $>1.$

Assuming first a positive answer to~\eqref{d.P^+xQ^+?}, I claim that
for all $P$ and $Q$ of cardinality $>1,$
\begin{equation}\begin{minipage}[c]{35pc}\label{d.P^+xQ^+=>}
$\dim(P)+\dim(Q) = \dim(P{_+^+})+\dim(Q{_+^+})
= \dim(P{_+^+}\times Q{_+^+})\,\leq\,
\dim(P_+\times Q_+) + 1\,\leq\,
\dim(P\times Q) + 2.$
\end{minipage}\end{equation}

Indeed, the first equality is seen by applying~\eqref{d.dim=}
to $P$ and $Q,$ the second follows
from~\eqref{d.KB}, the following inequality from the
assumed positive answer to~\eqref{d.P^+xQ^+?}, and the final
inequality by applying that same assumption to the opposites of
the posets in question.

Subtracting $2$ from both ends of~\eqref{d.P^+xQ^+=>}, we get
the inequality of~\eqref{d.leq2?}, establishing the conclusion of~(i)
under our assumption of a positive answer~\eqref{d.P^+xQ^+?}.

If $P$ and $Q$ each have a least element, we can
modify~\eqref{d.P^+xQ^+=>} by dropping all the
{\em subscripted} ``$\!+\!$"'s
and skipping the last step, yielding the conclusion of~(ii).

If, rather, we assume a positive answer to~\eqref{d.P_+xQ^+?},
we get a consequence like~\eqref{d.P^+xQ^+=>}, modified only by
changing the term $\dim(P_+\times Q_+)$ to $\dim(P^+\times Q_+).$
(The two final inequalities are
each gotten by applying~\eqref{d.P_+xQ^+?},
with the order of factors reversed in the last case.)
This again gives a positive answer to~\eqref{d.leq2?},
completing the proof of~(i),
while if $P$ has a least element and $Q$ a
greatest element, we can, as before, shorten our
version of~\eqref{d.P^+xQ^+=>}, and get~(iii).
\end{proof}

\section{Some observations on the above result}\label{S.re_P+}

I have repeatedly thought, ``Isn't the result asked for
in~\eqref{d.P^+xQ^+?} implied by that asked for in~\eqref{d.^+^+?},
or at least by that together with~\eqref{d.leq2?}, so that
question~\eqref{d.P^+xQ^+?} is just a reformulation
of the existing questions~\eqref{d.leq2?} and~\eqref{d.^+^+?}?''
But I do not see any way of getting the reverse implication
(nor the corresponding implication with~\eqref{d.P_+xQ^+?}
and~\eqref{d._+^+?} in place of~\eqref{d.P^+xQ^+?} and~\eqref{d.^+^+?}).
So it seems that the results asked for in~\eqref{d.P_+xQ^+?}
and~\eqref{d.P^+xQ^+?} are actually stronger than those asked for
in~\eqref{d.leq2?}-\eqref{d._+^+?}.

This does not mean that going from~\eqref{d.leq2?}-\eqref{d._+^+?}
to~\eqref{d.P^+xQ^+?} and~\eqref{d.P_+xQ^+?}
would necessarily make things harder.
At times, aiming for a stronger statement can lead one to
approaches that the question one started with would not have suggested.
On the down side, however,
questions~\eqref{d.P^+xQ^+?} and~\eqref{d.P_+xQ^+?}
are not as conceptually simple as~\eqref{d.leq2?}-\eqref{d._+^+?}.

One might think that alongside the inequalities
of~\eqref{d.P^+xQ^+?} and~\eqref{d.P_+xQ^+?} one could hope
for an inequality $\dim(P\times Q{_+^+})\leq \dim(P\times Q) + 1;$
but this does not hold.
For instance, if $P={\bf 2}\times{\bf 2},$ and $Q$ is a ``standard
example'' $S_n$ $(n>2)$~\cite[(1.6)]{poset_dim},
then $\dim(P\times Q) = \dim(P) = n$
\cite[(1.6)]{poset_dim}; but in view of~\eqref{d.KB},
$\dim(P\times Q{_+^+})=\dim(P)+\dim(Q{_+^+})
= \dim(P) + \dim(Q) = n+2 \not\leq n+1.$
% (cf.~\cite[paragraph following Question~3.2]{poset_dim}).

One might try to approach~\eqref{d.P^+xQ^+?} by regarding
$P^+\times Q^+$ as the union of the product-poset
$P\times Q$ with a poset adjoined ``above it'', namely
\begin{equation}\begin{minipage}[c]{35pc}\label{d.above}
$(P\times\{1\})\,\cup\,(\{1\}\times Q)\,\cup\,\{(1,1)\}.$
\end{minipage}\end{equation}
But that set lies ``above'' $P\times Q$ only in the weak sense
that none of its elements are $\leq$ any elements of $P\times Q.$
In general, some elements of~\eqref{d.above} are incomparable
with some elements of $P\times Q;$
and the fact~\cite[(1.11)]{poset_dim}
used in the proof of~\eqref{d.dim=} above, which applies when
a poset breaks into two parts, one of which is above
the other in the strong sense that {\em all} of its elements lie
above {\em all} elements of the other, is not true for this
weaker condition.
(For instance, for each $n\geq 3$ the standard example $S_n$
is the union of two $\!n\!$-element
antichains one of which ``lies above'' the other in the
weak sense; and each of those antichains has dimension $2,$ but
$S_n$ has dimension $n,$ which for large $n$ is far
greater than $\max(2,2).)$

We remark that if for some $n>1$ we could prove a weakened
version of the result asked for in~\eqref{d.P^+xQ^+?}
or~\eqref{d.P_+xQ^+?} with ``$\!+1\!$'' replaced by ``$\!+n\!$'',
then the method of proof of Proposition~\ref{P.=>}
would yield the corresponding conclusions with ``$\!+1\!$''
replaced by ``$\!+n\!$'', and ``$\!+2\!$'' by ``$\!+2n\!$''.
Likewise, if we ask
\begin{equation}\begin{minipage}[c]{35pc}\label{d.PxQ^+?}
Is there a positive integer $m$ such that, whenever
$P$ and $Q$ are posets of cardinality~$>1,$
$\dim(P\times Q^+)\,\leq\,\dim(P\times Q) + m$?
\end{minipage}\end{equation}
then it is easy to see that if there is such an $m,$ then
versions of~\eqref{d.P^+xQ^+?}
and~\eqref{d.P_+xQ^+?} with ``$\!+1\!$'' replaced by ``$\!+2m\!$''
will have positive answers, hence
the version of question~\eqref{d.leq2?}
with $2$ replaced by $4m$ would have a positive answer.
Even for $m=1$ this would be a considerable weakening of the
hoped-for inequality of~\eqref{d.leq2?}; but it would still
be a big improvement on our present lack of knowledge.
And the $m=1$ case of~\eqref{d.PxQ^+?} might be
easier to study than~\eqref{d.P^+xQ^+?} or~\eqref{d.P_+xQ^+?},
because of the simpler form of the poset it adjoins to $P\times Q.$

We also remark that it might be desirable to restrict
questions~\eqref{d.P^+xQ^+?} and~\eqref{d.P_+xQ^+?} to the
case where the posets $P$ and $Q$
are {\em connected} (Definition~\ref{D.main}).
As noted in~\cite[Question~4.4]{poset_dim} and the paragraphs
immediately preceding and following it,
and in~\S\ref{S.equiv} below, non-connected posets can
show some behaviors not known to occur in connected posets.
If those questions
could be answered affirmatively for connected posets,
positive answers to~\eqref{d.leq2?} etc.\ for connected
posets would follow, and from these, the corresponding results
for arbitrary posets could easily be deduced.

\section{Classifying posets by dimension behavior}\label{S.equiv}

In \cite[Definition~4.1]{poset_dim} the {\em absorbency} $\r{abs}(P)$
of a poset $P$ is defined to be the largest natural number
$n$ such that $\dim(P\times\prod_{i\in n} T_i) = \dim(P)$ for every
$\!n\!$-tuple of chains $(T_i)_{i\in n}.$
(As is implicit in question~\eqref{d.leq2?} above, the absorbency
is $0,$ $1$ or $2$ in all cases where we know its value.)
Given a poset $P,$ the data about $P$ that we presently
know how to use in estimating the
dimensions of products $P\times Q$
are its dimension and its absorbency.
Can we associate more such information to a poset $P$?

Here is some notation and language we shall use in examining
this question.

\begin{definition}\label{d.dim-beh}
Given posets $P$ and $P',$ we shall write $P\preccurlyeq P'$ if for
all posets $Q,$ one has $\dim(P\times Q)\leq\dim(P'\times Q).$
If $P\preccurlyeq P'$ and $P'\preccurlyeq P,$ we shall write
$P\approx P'.$

If a poset $P$ is partitioned into disjoint subposets
$P_0,\dots,P_{n-1}$ such that for $i\neq j,$ no element of
$P_i$ is comparable with any element of $P_j,$ we shall
write $P=P_0\,\sqcup\,\dots\,\sqcup\,P_{n-1}.$
Likewise, given arbitrary posets $P_0,\dots,P_{n-1},$ we shall
write $P_0\,\sqcup\,\dots\,\sqcup\,P_{n-1}$ for the poset whose
underlying set is the union of the underlying sets
of $\{0\}\times P_0,\dots,\{n\!-\!1\}\times P_{n-1},$ and
where elements $(i,p)\in\{i\}\times P_i$ and
$(j,p')\in\{j\}\times P_j$ satisfy
$(i,p)\leq (j,p')$ if and only if $i=j$ and $p\leq p'$ in $P_i.$

As noted in Definition~\ref{D.main},
a poset $P$ will be called {\em connected} if it has no
expression $P = P'\,\sqcup\,P''$ for subposets $P'$ and $P''$
{\rm(}understood, as always, to be nonempty{\rm)}.
% $P$ will be called {\em non-connected} otherwise.
\end{definition}

So, for instance, \cite[Question~3.4]{poset_dim} asks whether
all chains of more than one element are  $\!\approx\!$-equivalent.

In fact, I do not know the answer to either of the following
two questions, which point to opposite possible extremes:
\begin{equation}\begin{minipage}[c]{35pc}\label{d.dim+ab?}
Are every pair of posets having equal dimension
and equal absorbency $\!\approx\!$-equivalent?
\end{minipage}\end{equation}
\begin{equation}\begin{minipage}[c]{35pc}\label{d.anyiso?}
Are $\!\approx\!$-equivalent connected posets always isomorphic?
\end{minipage}\end{equation}

For the relation $\preccurlyeq,$ one at least has a wider
class of examples known to satisfy $P\preccurlyeq Q$: Clearly,
any subposet $P$ of a poset $Q$ satisfies that relation.

For non-connected posets, one can give much larger classes of
examples of these relations.
We first note

\begin{lemma}\label{L.ncnctd_prod}
Given a poset $P=P_0\,\sqcup\,\dots\,\sqcup\,P_{n-1}$
where $n\geq 2,$ and any poset $Q,$ we have
\begin{equation}\begin{minipage}[c]{35pc}\label{d.noncnctd_prod}
$\dim(P\times Q)\,=\,
\max(2,\,\dim(P_0\times Q),\,\dots,\,\dim(P_{n-1}\times Q)).$
\end{minipage}\end{equation}
\end{lemma}

\begin{proof}
Note that
$P\times Q = (P_0\times Q)\,\sqcup\,\dots\,\sqcup\,(P_{n-1}\times Q),$
and apply~\cite[(1.11)]{poset_dim} with the antichain
$\{0,\dots,n-1\},$ which has dimension~$2,$
in the role of the indexing poset $P$ of that display, and
the products $P_i\times Q$ $(i=0,\dots,n{-}1)$
in the roles of the posets there denoted $Q_p.$
\end{proof}

We can now deduce the following.
(In the statement below, we could drop any or all of the
restrictions ``$\!\preccurlyeq\!$-maximal'', giving a formally
simpler statement.
But including those conditions reduces the set of cases that
must be checked in determining whether a pair of posets
is $\!\approx\!$-equivalent, and gives a clearer picture
of the $\!\approx\!$-equivalent families described.)

\begin{proposition}\label{P.max_Pi}
Suppose $P=P_0\,\sqcup\,\dots\,\sqcup\,P_{n-1}$ and
$P'=P'_0\,\sqcup\,\dots\,\sqcup\,P'_{n'-1}$ are posets, each having
dimension $\geq~2.$

Then a sufficient condition for $P\approx P'$ to hold
is that every $\!\preccurlyeq\!$-maximal member $P_i$ of
the family $P_0,\dots,P_{n-1}$ is $\!\approx\!$-equivalent
to a $\!\preccurlyeq\!$-maximal member $P'_{i'}$ of
the family $P'_0,\dots,P'_{n'-1},$ and
vice versa.
\end{proposition}

\begin{proof}
Let us first note that for $P=P_0\,\sqcup\,\dots\,\sqcup\,P_{n-1}$
having dimension $\geq 2$ as
assumed here, and any $Q,$~\eqref{d.noncnctd_prod} holds
without the requirement of Lemma~\ref{L.ncnctd_prod} that $n\geq 2.$
Indeed, if $n=1,$ $\dim(P\times Q)= \dim(P_0\times Q),$
and our hypothesis that $P$ has dimension $\geq 2$ implies that
$\dim(P\times Q)\geq 2,$ so we indeed have
$\dim(P\times Q)=\max(2,\,\dim(P_0\times Q)),$
i.e.,~\eqref{d.noncnctd_prod}.

I claim that the resulting formula yields the same values for
$\dim(P\times Q)$ and $\dim(P'\times Q).$

To show that the former is $\leq$ the latter,
it will suffice to show that the largest of the
values $\dim(P_i\times Q)$ is also one of the
values $\dim(P'_{i'}\times Q).$
The former largest value will be achieved at some
$P_i$ which is $\!\preccurlyeq\!$-maximal among the $P_j,$
and by hypothesis, there is an $\!\approx\!$-equivalent $P'_{i'}$
among the $P'_{j'}.$
By this $\!\approx\!$-equivalence we have
$\dim(P'_{i'}\times Q)=\dim(P_i\times Q),$ hence
applying~\eqref{d.noncnctd_prod}
we indeed have $\dim(P\times Q)\leq\dim(P'\times Q).$

By symmetry we also have $\dim(P'\times Q)\leq\dim(P\times Q),$
completing the proof.
\end{proof}

So, for example, starting with any
$P=P_0\,\sqcup\,\dots\,\sqcup\,P_{n-1}$
of dimension $\geq 2,$ we can get infinitely many
isomorphism classes of posets $P'\approx P$
by throwing in multiple copies of subposets of the $P_i.$

We remark that the condition that $P$ and $P'$ each have
dimension~$\geq 2$ in the hypothesis of Proposition~\ref{P.max_Pi}
cannot be dropped.
For instance, suppose that $P={\bf 2}$ and $P'={\bf 2}\sqcup{\bf 2}.$
The $\!\preccurlyeq\!$-maximal pieces in the above descriptions
of $P$ and $P'$ are certainly $\!\approx\!$-equivalent,
but $P$ fails to have dimension $\geq 2;$ so
letting $Q$ be a $\!1\!$-element poset, we see that
$\dim(P\times Q)=\dim(P)=1,$ while $\dim(P'\times Q)=\dim(P')=2.$

I don't know whether, for every pair of
posets $P$ and $P'$ and pair of decompositions of
these posets as in Proposition~\ref{P.max_Pi}, with the
added condition that all $P_i$ and $P'_{i'}$ be connected,
the sufficient condition of that Proposition
for $P$ and $P'$ to be $\!\approx\!$-equivalent is also necessary.
A couple of cases of this question are:
\begin{equation}\begin{minipage}[c]{35pc}\label{q.P0xP1=Q0xQ1?}
Suppose $P_0,$ $P_1,$ $P'_0,$ $P'_1$ are connected posets
satisfying $P_0\,\sqcup\,P_1\approx P'_0\,\sqcup\,P'_1,$
and such that $P_0$ is $\!\preccurlyeq\!$-incomparable with $P_1,$
and $P'_0$ is $\!\preccurlyeq\!$-incomparable with $P'_1.$
Must we have either $P_0\approx Q_0$ and $P_1\approx Q_1,$ or
$P_0\approx Q_1$ and $P_1\approx Q_0$?
\end{minipage}\end{equation}
\begin{equation}\begin{minipage}[c]{35pc}\label{d.2=1}
Can one have a relation $P\approx P'_0\,\sqcup\,P'_1$ where
$P$ is a connected poset, and neither $P'_0$ nor $P'_1$
is $\approx P$?
\end{minipage}\end{equation}

Returning to Proposition~\ref{P.max_Pi}, the reader can easily verify
the corresponding criterion for $P\preccurlyeq P'$ to hold:

\begin{corollary}[to
proof of Proposition~\ref{P.max_Pi}]\label{C.max_Pi}
Suppose $P=P_0\,\sqcup\,\dots\,\sqcup\,P_{n-1}$ and
$P'=P'_0\,\sqcup\,\dots\,\sqcup\,P'_{n'-1}$ are posets,
with $\dim(P')\geq 2.$

Then a sufficient condition for $P\preccurlyeq P'$ to hold
is that for every $\!\preccurlyeq\!$-maximal member $P_i$ of
the family $P_0,\dots,P_{n-1},$ one has $P_i\preccurlyeq P'_{i'}$
for some $\!\preccurlyeq\!$-maximal member $P'_{i'}$ of
the family $P'_0,\dots,P'_{n'-1}.$\qed
\end{corollary}

% \section{Acknowledgements}\label{S.ackn}


\begin{thebibliography}{00}

\bibitem{KB} Kirby A. Baker,
{\em Dimension, join-independence, and breadth in partially ordered
sets}, Honors Thesis, Harvard University, 1961 (unpublished), 28\,pp.
\url{https://www.math.ucla.edu/~baker/res/po/undergraduate_thesis.pdf}

\bibitem{poset_dim} George M. Bergman,
{\em Some frustrating questions on dimensions of products of posets},
Discrete Mathematics, {\bf 349} Issue 6, June 2026, article 115002.
arXiv:2312.12615.  MR5019679.

\bibitem{K+T} David Kelly and William T. Trotter, Jr.,
{\em Dimension theory for ordered sets}, pp.\,171-211 in
{\em Ordered sets \textup{(}Banff, Alta, 1981\textup{)}},
NATO Adv. Study Inst. Ser. C: Math. Phys. Sci.,
{\bf 83} (1982).
MR0661294

\end{thebibliography}
\end{document}